\documentclass[a4paper,10pt]{article}

\usepackage{amscd,amssymb,amsmath,graphicx,verbatim,hyperref}
\usepackage[OT1,T1]{fontenc}

\usepackage[all]{xypic}
\usepackage{pdfpages}

\newcommand{\qed}{\nobreak \ifvmode \relax \else
      \ifdim\lastskip<1.5em \hskip-\lastskip
      \hskip1.5em plus0em minus0.5em \fi \nobreak
      \vrule height0.75em width0.5em depth0.25em\fi}

\newcounter{Definitioncount}
\newtheorem{theorem}{Theorem}

\newtheorem{proposition}[theorem]{Proposition}
\newtheorem{corollary}[theorem]{Corollary}

\newenvironment{proof}[1][Proof]{\begin{trivlist}
\item[\hskip \labelsep {\bfseries #1}]}{\end{trivlist}}
\newenvironment{definition}[1][Definition]{\begin{trivlist}
\item[\hskip \labelsep {\bfseries #1}] \refstepcounter{Definitioncount} \textbf{\arabic{Definitioncount}} }{\end{trivlist}}

\newcommand{\bc}{{\cal{C}}}
\newcommand{\ba}{{\cal{A}}}
\newcommand{\bk}{{\cal{K}}}

\newcommand{\bx}{{\cal{X}}}

\newcommand{\ev}{\mathcal{V} }

\numberwithin{equation}{section}

\begin{document}

\author{Ross Street\footnote{The author gratefully acknowledges the support of an Australian Research Council Discovery Grant DP1094883.}}
\title{The core of adjoint functors}
\date{}
\maketitle

\begin{center}
--------------------------------------------------------
\end{center}

\begin{abstract}
\noindent There is a lot of redundancy in the usual definition of adjoint functors. We define and prove the core of what is required. First we do this in the hom-enriched context. Then we do it in the cocompletion of a bicategory with respect to Kleisli objects, which we then apply to internal categories. Finally, we describe a doctrinal setting. 
\end{abstract}
\tableofcontents

\section{Introduction}

Kan [\ref{Kan1958}] introduced the notion of adjoint functors. By defining the unit and counit natural transformations, he paved the way for the notion to be internalized to any 2-category. This was done by Kelly [\ref{Kelly1969}] whose interest at the time was particularly in the 2-category of $\ev$-categories for a monoidal category $\ev$ (in the sense of Eilenberg-Kelly [\ref{EilKel1966}]). 

During my Topology lectures at Macquarie University in the 1970s, the students and I realized, in proving that a function $f$ between posets was order preserving when there was a function $u$ in the reverse direction such that $f(x) \le a$ if and only if $x \le u(a)$, did not require $u$ to be order preserving. I realized then that knowing functors in the two directions only on objects and the usual hom adjointness isomorphism implied the effect of the functors on homs was uniquely determined. Writing this down properly led to the present paper. 

Section \ref{adj} merely reviews adjunctions between enriched categories. Section \ref{Crs}  introduces the notion of core of an enriched adjunction: it only involves the object assignments of the two functors and a hom isomorphism with no naturality requirement. The main result characterizes when such a core is an adjunction. 

The material becomes increasingly for mature audiences; that is, for those with knowledge of bicategories. Sections \ref{abm} and \ref{cbm} present results about adjunctions in the Kleisli object cocompletion of a bicategory in the sense of [\ref{FTMII}]. In particular, this is applied in Section \ref{cbic} to adjunctions for categories internal to a finitely complete category. By a different choice of bicategory, where enriched categories can be seen as monads (see [\ref{BCSW}]), we could rediscover the work of Section \ref{Crs}; however, we leave this to the interested reader.  In Section \ref{doctrinal} we describe a general setting, involving a pseudomonad (doctrine) on a bicategory, using a construction of Mark Weber [\ref{WeberACS}].

\section{Adjunctions}\label{adj}

For $\ev$-categories $\ba$ and $\bx$, an
{\it adjunction} consists of
\begin{enumerate}
\item  $\ev$-functors $U:\ba\longrightarrow \bx$
and $F:\bx\longrightarrow \mathcal{A}$;
\item a $\ev$-natural family of isomorphisms $\pi :\mathcal{A}(
F X,A) \cong \bx( X,U A) $ in $\ev$ indexed by $A\in
\mathcal{A}$, $X\in \bx$.
\end{enumerate}
We write $\pi :F\dashv U:\mathcal{A}\longrightarrow \bx$.

The following result is well known; for example see Section 1.11
of \cite{KellyBook}.
\begin{proposition}
Suppose $U:\mathcal{A}\longrightarrow \bx$ is a $\ev$-functor,
$F:\mathit{ob}\bx\longrightarrow \mathit{ob}\mathcal{A}$
is a function, and, for each $X\in \bx$, $\pi :\mathcal{A}(
F X,A) \cong \bx( X,U A) $ is a family of isomorphisms $\ev$-natural
in $A\in \mathcal{A}$. Then there exists a unique adjunction $\pi
:F\dashv U:\mathcal{A}\longrightarrow \bx$ for which $F:\mathrm{ob\bx}\longrightarrow
\mathrm{ob\mathcal{A}}$ is the effect of the $\ev$-functor
$F:\bx\longrightarrow \mathcal{A}$ on objects. \qed \ \ \ \ \ \ \label{XRef-Proposition-121173012}
\end{proposition}

\section{Cores}\label{Crs}

\begin{definition}
For $\ev$-categories $\mathcal{A}$ and $\bx$, an
{\it adjunction core} consists of
\begin{enumerate}
\item functions $U:\mathit{ob}\mathcal{A}\longrightarrow \mathit{ob}\bx$
and $F:\mathit{ob}\bx\longrightarrow \mathit{ob}\mathcal{A}$;\ \ 
\item a family of isomorphisms $\pi :\mathcal{A}( F X,A) \cong \bx(
X,U A) $ in $\ev$ indexed by $A\in \mathcal{A}$, $X\in \bx$.
\end{enumerate}
\end{definition}

Given such a core, we make the following definitions:

\vspace*{5mm}

(a)  $\beta _{X}:X\longrightarrow U F X$ is the composite
\[I\overset{j}{\longrightarrow }\mathcal{A}( F X,F X) \overset{\pi}{\longrightarrow }\bx( X,U F X) ;\]

(b)  $\alpha _{A}:F U A\longrightarrow A$ is the composite\[I\overset{j}{\longrightarrow }\bx( U A,U A) \overset{\pi^{-1}}{\longrightarrow }\mathcal{A}( F U A,X) ;\]

(c)  $U_{A B}:\mathcal{A}( A,B) \longrightarrow \bx( UA,U B) $ is the composite\[\mathcal{A}( A,B) \overset{\mathcal{A}( \alpha _{A},1) \ \ }{\longrightarrow}\mathcal{A}( F U A,B) \overset{\pi }{\longrightarrow }\bx(U A,U B) ;\]

(d)  $F_{X Y}:\bx( X,Y) \longrightarrow \mathcal{A}( FX,F Y) $ is the composite \[\bx( X,Y) \overset{\bx( 1,\beta _{Y}) \ \ }{\longrightarrow}\bx( X,U F Y) \overset{\pi ^{-1}}{\longrightarrow }\mathcal{A}(F X,F Y) .\]

Clearly each adjunction $\pi :F\dashv U:\mathcal{A}\longrightarrow\bx$ includes an adjunction core as part of its data. Then it follows directly from the Yoneda lemma and the definitions (a) and (b) that the effect of $U$ and $F$ on homs are as in (c) and (d).

\begin{theorem} \label{enrichedCoreThm}
An adjunction core extends to an adjunction if and only if one of
the diagrams (\ref{XRef-Equation-121162956}) or (\ref{XRef-Equation-121163025})
below commutes. The adjunction is unique when it exists.

\begin{equation}\label{XRef-Equation-121162956}
\xymatrix{\ba (A,B) \otimes \ba (FX,A) \ar[rr]^-{U_{A,B}\otimes\pi} \ar[d]_-{\mathrm{comp}} && \bx (UA,UB) \otimes \bx (X, UA) \ar[d]^-{\mathrm{comp}} \\\ba (FX,B) \ar[rr]_-{\pi} && \bx (X,UB)}
\end{equation}

\begin{equation}\label{XRef-Equation-121163025}
\xymatrix{
\bx (Y,UA) \otimes \bx (X,Y) \ar[rr]^-{\pi^{-1}\otimes F_{X,Y}} \ar[d]_-{\mathrm{comp}} && \ba (FY,A) \otimes \ba (FX, FY) \ar[d]^-{\mathrm{comp}} \\
\bx (X,UA) \ar[rr]_-{\pi^{-1}} && \ba (FX,A)}
\end{equation}
\end{theorem}

\begin{proof} 
We deal first with the version involving diagram
(\ref{XRef-Equation-121162956}). For an adjunction, (\ref{XRef-Equation-121162956})
expresses the $\ev$-naturality of $\pi $ in $A\in \mathcal{A}$.
Conversely, given an adjunction core satisfying (\ref{XRef-Equation-121162956}),
we paste to the left of (\ref{XRef-Equation-121162956} with $X =UC$, the diagram 
\begin{equation}
\xymatrix{
\ba (A,B) \otimes \ba (C,A) \ar[rr]^-{1 \otimes \ba (\alpha _C ,1) } \ar[d]_-{\mathrm{comp}} && \ba (A,B) \otimes \ba (FUC,A) \ar[d]^-{\mathrm{comp}} \\
\bx (X,UA)  \ar[rr]_-{\pi^{-1} } && \ba (FX,A) }
\end{equation}
which commutes by naturality of composition. This leads to the following commutative square.
\begin{equation}\label{XRef-Equation-121172723}
\xymatrix{
\ba (A,B) \otimes \ba (C,A)  \ar[rr]^-{U \otimes U} \ar[d]_-{\mathrm{comp}} && \bx (UA,UB) \otimes \bx (UC,UA)  \ar[d]^-{\mathrm{comp}} \\
\ba (C,B)  \ar[rr]_-{U} && \bx (UC,UB) }
\end{equation}
We also have the equality
\begin{equation}
\left( I\overset{j}{\longrightarrow }\mathcal{A}( A,A) \overset{U}{\longrightarrow
}\bx( U A,U B) \right) =\left( I\overset{j}{\longrightarrow
}\bx( U A,U B) \right) %
\label{XRef-Equation-121172748}
\end{equation}
straight from the definitions (b) and (c). Together (\ref{XRef-Equation-121172723})
and (\ref{XRef-Equation-121172748}) tell us that $U$ is a $\ev$-functor.
Now the general diagram (\ref{XRef-Equation-121162956}) expresses
the $\ev$-naturality of $\pi $ in $A$. By Proposition \ref{XRef-Proposition-121173012},
we have an adjunction determined uniquely by the core.

Writing $\ev^{\mathrm{rev}}$ for $\ev$ with the
reversed monoidal structure $A\otimes ^{\mathrm{rev}}B=B\otimes
A$, and applying the first part of this proof to the $\ev^{\mathrm{rev}}$-enriched
adjunction $\pi ^{-1}:U^{\mathrm{op}}\dashv F^{\mathrm{op}}:\bx^{\mathrm{op}}\longrightarrow
\mathcal{A}^{\mathrm{op}}$, which is the same as an adjunction $\pi
:F\dashv U:\mathcal{A}\longrightarrow \bx$, we see that
it is equivalent to an adjunction core satisfying (\ref{XRef-Equation-121163025}). \qed
\end{proof}

\begin{corollary}
If $\ev$ is a poset then adjunction cores are adjunctions.
\end{corollary}
\begin{proof}
All diagrams, including (\ref{XRef-Equation-121162956}),
commute in such a $\ev$. \qed
\end{proof}

\section{Adjunctions between monads}\label{abm} 

This section will discuss adjunctions in a particular bicategory 
$\mathrm{KL} (\bk)$ of monads in a bicategory $\bk$. 
The results will apply to adjunctions between categories internal to a category $\bc$ with pullbacks.

As well as defining bicategories B\'enabou [\ref{Ben1967}] defined, 
for each pair of bicategories $\ba$ and $\bk$, a bicategory $\mathrm{Bicat} (\ba,\bk)$ 
whose objects are morphisms $\ba \longrightarrow \bk$ of bicategories (also called lax functors), whose morphisms are transformations (also called lax natural transformations), 
and whose 2-cells are modifications. 
In particular, $\mathrm{Bicat} (\bf{1} ,\bk)$ is one bicategory whose objects are monads in $\bk$; 
it was called $\mathrm{Mnd}(\bk)$ in [\ref{FTM}] for the case of a 2-category $\bk$, 
where it was used to discuss Eilenberg-Moore objects in $\bk$. 
We shall also use the notation $\mathrm{Mnd}(\bk)$ when $\bk$ is a bicategory.  

We write $\bk^{\mathrm{op}}$ for the dual of $\bk$ obtained by reversing morphisms (not 2-cells). 
Monads in $\bk^{\mathrm{op}}$ are the same as monads in $\bk$. 
So we also have the bicategory 
$$\mathrm{Mnd}^{\mathrm{op}}(\bk) = \mathrm{Bicat} (\bf{1} ,\bk^{\mathrm{op}})^{\mathrm{op}}$$
whose objects are monads in $\bk$. This was used in [\ref{FTM}] to discuss Kleisli objects in $\bk$.

Two more bicategories $\mathrm{EM} (\bk)$ and $\mathrm{KL} (\bk)$, with objects monads in $\bk$, were defined in [\ref{FTMII}]. 
The first freely adjoins Eilenberg-Moore objects and the second freely adjoins Kleisli objects to $\bk$. 
In fact, $\mathrm{EM} (\bk)$ has the same objects and morphisms as $\mathrm{Mnd}(\bk)$ 
but different 2-cells while $\mathrm{KL} (\bk)$ has the same objects and morphisms as 
$\mathrm{Mnd}^{\mathrm{op}}(\bk)$ but different 2-cells.  

A {\it monad} in a bicategory $\bk$ is an object $A$ equipped with a morphism 
$s:A \longrightarrow A$ and 2-cells $\eta : 1_A \longrightarrow s$ and $\mu : ss \longrightarrow s$ 
such that 
\begin{equation}
\xymatrix{
& s(ss) \ar[rd]^-{ s \mu}  & \\
(ss)s \ar[ru]^-{\cong} \ar[d]_-{\mu s} & & ss \ar[d]^-{\mu} \\
ss \ar[rr]_-{\mu} & & w }
\end{equation}

\noindent and the composites
\begin{equation}
s 1\overset{s \eta }{\longrightarrow }s s\overset{\mu }{\longrightarrow
}s\ \ \ \ \ \mathrm{and}\ \ \ \ \ 1 s\overset{\eta  s}{\longrightarrow
}s s\overset{\mu }{\longrightarrow }s
\end{equation}
should be the canonical isomorphisms. 
We shall use the same symbols $\eta$ and $\mu$ for the {\it unit} and {\it multiplication} of all monads; so we simply write $(A,s)$ for the monad.

For monads $(A,s)$ and $(A^{\prime },s^{\prime })$ in $\bk$, a {\it monad opmorphism} $(f,\phi):(A,s) \longrightarrow (A^{\prime },s^{\prime })$ consists of a morphism $f:A \longrightarrow A^{\prime }$ and a 2-cell $\phi: fs \longrightarrow s^{\prime } f$ in $\bk$ such that
\begin{multline}\label{XRef-Equation-opmor1}
\left( \left( f s\right) s\overset{\phi  s}{\longrightarrow }\left(
s^{\prime }f\right) s\overset{\cong }{\longrightarrow }s^{\prime
}( f s) \overset{s^{\prime }\phi }{\longrightarrow }s^{\prime }(
s^{\prime } f) \overset{\cong }{\longrightarrow }s^{\prime }(  s^{\prime
}f) \overset{\mu  f}{\longrightarrow }s^{\prime }f\right) \\
=\left( \left( f s\right) s\overset{\cong }{\longrightarrow }f(
s s) \overset{f \mu }{\longrightarrow }f s\overset{\phi }{\longrightarrow
}s^{\prime } f\right) 
\end{multline}

\noindent and
\begin{equation}\label{XRef-Equation-opmor2}
\left( f 1\overset{f \eta }{\longrightarrow }f s\overset{\phi }{\longrightarrow
}s^{\prime } f\right) =\left( f 1\overset{\cong }{\longrightarrow
}1 f\overset{\eta  f}{\longrightarrow }s^{\prime } f\right)  .
\end{equation}
\noindent The composite of monad opmorphisms 
$(f,\phi):(A,s) \longrightarrow (A^{\prime },s^{\prime })$ 
and $(f^{\prime },\phi ^{\prime }):(A^{\prime },s^{\prime }) \longrightarrow (A^{\prime\prime },s^{\prime\prime })$ 
is defined to be $(f^{\prime }f,\phi ^{\prime } \star \phi):(A,s) \longrightarrow (A^{\prime\prime },s^{\prime\prime })$ 
where $\phi ^{\prime } \star \phi$ is the composite 
\begin{equation}
\left( f^{\prime }f\right) s\overset{\cong }{\longrightarrow }f^{\prime
}( f s) \overset{f^{\prime }\phi }{\longrightarrow }f^{\prime }(
s^{\prime } f) \overset{\cong }{\longrightarrow }\left( f^{\prime
} s^{\prime }\right) f\overset{\phi ^{\prime }f}{\longrightarrow
}\left( s^{{\prime\prime}}f^{\prime }\right) f\overset{\cong }{\longrightarrow
}s^{{\prime\prime}}( f^{\prime }f) \ \ .
\end{equation}

The objects of both $\mathrm{Mnd}^{\mathrm{op}}(\bk)$ and $\mathrm{KL}(\bk)$ are monads $(A,s)$ in $\bk$. 
The morphisms in both are the opmorphisms $(f, \phi)$. 
The 2-cells $\sigma : (f, \phi)\longrightarrow (g, \psi) : (A,s) \longrightarrow (A^{\prime},s^{\prime})$ 
in  $\mathrm{Mnd}^{\mathrm{op}}(\bk)$ are 2-cells $\sigma : f \longrightarrow g$ in $\bk$ 
such that the following square commutes.  
\begin{equation}
\xymatrix{
fs \ar[rr]^-{\phi} \ar[d]_-{\sigma s} && s^{\prime}f \ar[d]^-{s^{\prime} \sigma} \\
gs \ar[rr]_-{\psi} && s^{\prime}g}
\end{equation}

\noindent Vertical and horizontal composition in $\mathrm{Mnd}^{\mathrm{op}}(\bk)$ 
are performed in the obvious way so that the projection 
$\mathrm{Und} : \mathrm{Mnd}^{\mathrm{op}}(\bk) \longrightarrow \bk$ 
taking $(A,s)$ to $A$, $(f,\phi)$ to $f$, and $\sigma$ to $\sigma$. 
The associativity and unit isomorphisms in $\mathrm{Mnd}^{\mathrm{op}}(\bk)$ 
are also such that $\mathrm{Und}$ preserves them, making 
$\mathrm{Und}$ a strict morphism of bicategories.

A 2-cell $\rho : (f, \phi)\longrightarrow (g, \psi) : (A,s) \longrightarrow (A^{\prime},s^{\prime})$ in $\mathrm{KL}(\bk)$ is a 2-cell $\rho : f \longrightarrow s^{\prime} g$ in $\bk$ such that  
\begin{multline}\label{KL2cell}
\left( f s\overset{\phi }{\longrightarrow }s^{\prime }f \overset{s^{\prime
}\rho }{\longrightarrow }s^{\prime }( s^{\prime } g) \overset{\cong
}{\longrightarrow }\left( s^{\prime } s^{\prime }\right) g\overset{\mu
g}{\longrightarrow }s^{\prime }g\right) \\
=\left( f s\overset{\rho  s}{\longrightarrow }\left( s^{\prime }
g\right) s\overset{\cong }{\longrightarrow }s^{\prime } \left( g
s\right) \overset{s^{\prime } \psi }{\longrightarrow }s^{\prime
} \left( s^{\prime } g\right) \overset{\cong }{\longrightarrow }\left(
s^{\prime } s^{\prime }\right)  g\overset{\mu  g}{\longrightarrow
}s^{\prime } g\right) .
\end{multline}
\noindent The vertical composite of the 2-cells $\rho:(f,\phi) \longrightarrow (g,\psi)$ and $\tau:(g,\psi)\longrightarrow (h,\theta)$ is the 2-cell
\begin{equation}
f\overset{\rho }{\longrightarrow }s^{\prime }g\overset{s^{\prime
}\tau }{\longrightarrow }s^{\prime }( s^{\prime }h) \overset{\cong
}{\longrightarrow }\left( s^{\prime }s^{\prime }\right) h\overset{\mu
h}{\longrightarrow }s^{\prime }h\ \ .
\end{equation}
\noindent The horizontal composite of 2-cells $\rho : (f, \phi)\longrightarrow (g, \psi) : (A,s) \longrightarrow (A^{\prime},s^{\prime})$ and $\rho^{\prime} : (f^{\prime}, \phi^{\prime})\longrightarrow (g^{\prime}, \psi^{\prime}) : (A^{\prime},s^{\prime}) \longrightarrow (A^{\prime \prime},s^{\prime \prime})$ is the 2-cell
\begin{multline}
f^{\prime }f\overset{f^{\prime }\rho }{\longrightarrow }f^{\prime
}( s^{\prime }g) \overset{\cong }{\longrightarrow } \left( f^{\prime
}s^{\prime }\right) g\overset{\phi ^{\prime }g}{\longrightarrow
}\left( s^{{\prime\prime}}f^{\prime }\right) g\\
\overset{\left( s^{{\prime\prime}}\rho ^{\prime }\right) g}{\longrightarrow
}\left( s^{{\prime\prime}}( s^{{\prime\prime}}g^{\prime }) \right)
g\overset{\cong }{\longrightarrow }\left( s^{{\prime\prime}}s^{{\prime\prime}}\right)
\left( g^{\prime }g\right) \overset{\mu  \left( g^{\prime }g\right)
}{\longrightarrow }s^{{\prime\prime}}( g^{\prime }g) \ \ .
\end{multline}
\noindent Each 2-cell $\sigma :(f,\phi) \longrightarrow (g,\psi)$ 
in $\mathrm{Mnd}^{\mathrm{op}}(\bk)$ defines a 2-cell 
$\rho : (f, \phi) \longrightarrow (g, \psi)$ 
in $\mathrm{KL}(\bk)$ via $\rho = \eta g \cdot \sigma$.
The associativity and unit isomorphisms for $\mathrm{KL}(\bk)$ are
determined by the condition that we have a strict morphism of bicategories 
$$K: \mathrm{Mnd}^{\mathrm{op}}(\bk) \longrightarrow \mathrm{KL}(\bk)$$
which is the identity on objects and morphisms and takes each 2-cell $\sigma$
to $\eta g \cdot \sigma$. 

Henceforth we shall invoke the coherence theorem (see [\ref{MacLPar}] and [\ref{GPS}]) that every bicategory is biequivalent to a 2-category to write as if we were working in a 2-category
$\mathrm{KL}(\bk)$. We also recommend reworking the proofs below using the string diagrams of [\ref{GTC}] as adapted for bicategories in [\ref{HCSCSE}] and [\ref{CatStr}].

Now we are in a position to examine what is involved in an adjunction
\begin{equation}\label{AdjKL}
\left( f,\phi \right) \dashv \left( u,\upsilon \right) :\left( A,s\right)
\longrightarrow \left( X,t\right) 
\end{equation}
with counit $\alpha: (f,\phi)\cdot(u,\upsilon)\longrightarrow 1_{(A,s)}$ 
and unit $\beta: 1_{(X,t)} \longrightarrow (u,\upsilon)\cdot(f,\phi)$ in $\mathrm{KL}(\bk)$. 

We have morphisms $u:A \longrightarrow X$ and $f:X \longrightarrow A$ in $\bk$.
We have 2-cells $\upsilon : uf \longrightarrow tu$ 
and $\phi : ft \longrightarrow sf$ both satisfying (\ref{XRef-Equation-opmor1}) 
and (\ref{XRef-Equation-opmor2}) with the variables appropriately substituted.

We have a 2-cell $\alpha: fu \longrightarrow s$ satisfying
\begin{equation}\label{alpha}
\left( f u s\overset{\alpha  s}{\longrightarrow }s s\overset{\mu
}{\longrightarrow }s\right) =\left( f u s \overset{f \upsilon }{\longrightarrow
}f t u\overset{\phi  u}{\longrightarrow }s f u\overset{s \alpha
}{\longrightarrow }s s\overset{\mu }{\longrightarrow }s\right) 
\end{equation} 
which is (\ref{KL2cell}) for $\alpha$. 

We have a 2-cell $\beta : 1_X \longrightarrow tuf$ satisfying
\begin{equation}\label{beta}
\left( t\overset{t \beta }{\longrightarrow }t t u f\overset{\mu
u f}{\longrightarrow }t u f\right) =\left( t \overset{\beta  t}{\longrightarrow
}t u f t\overset{t u \phi }{\longrightarrow }t u s f\overset{t \upsilon
f}{\longrightarrow }t t u f\overset{\mu  u f}{\longrightarrow }t
u f\right) 
\end{equation}
which is (\ref{KL2cell}) for $\beta$. 

Using the rules for compositions in $\mathrm{KL}(\bk)$, 
we see that the two triangle conditions for the counit and unit of 
an adjunction become, in this case, the identities
\begin{equation}\label{triangle1}
\left( f\overset{\eta  f}{\longrightarrow }s f\right) =\left( f
\overset{f \beta }{\longrightarrow} ftuf \overset{\phi  u f}{\longrightarrow
}s f u f\overset{s \alpha  f}{\longrightarrow }s s f\overset{\mu
f}{\longrightarrow }s f\right) 
\end{equation}
and
\begin{equation}\label{triangle2}
\left( u\overset{\eta  u}{\longrightarrow }t u\right) =\left( u
\overset{\beta  u}{\longrightarrow }t u f u\overset{t u \alpha }{\longrightarrow
}t u s\overset{t \upsilon }{\longrightarrow }t t u\overset{\mu 
u}{\longrightarrow }t u\right) \ \ .
\end{equation}

It is common to call a morphism $f:X \longrightarrow A$ in a bicategory $\bk$ 
a \textit{map} when it has a right adjoint. We write $f^{\star} : A \longrightarrow X$
for a selected right adjoint, $\eta_f :1_X \longrightarrow f^{\star} f$ for the unit, and $\varepsilon_f :f f^{\star} \longrightarrow 1_A$ for the counit.

\begin{theorem}\label{ThmAdjKL}
Suppose (\ref{AdjKL}) is an adjunction in $\mathrm{KL}(\bk)$ with counit $\alpha$ and unit $\beta$, and suppose $f:X \longrightarrow A$ is a map in $\bk$. Then the composite 2-cell 
\begin{equation} \label{pi}
\pi  : f^{\star}s\overset{\beta  f^{\star}s}{\longrightarrow }t
u f f^{\star}s\overset{t u \varepsilon _{f}s}{\longrightarrow }t u s
\overset{t \upsilon }{\longrightarrow }t t u\overset{\mu  u}{\longrightarrow}t u
\end{equation}
in $\bk$ is invertible with inverse defined by the composite 2-cell
\begin{equation} \label{pi-inv}
\pi ^{-1} : t u\overset{\eta _{f} t u}{\longrightarrow }f^{\star}f t u\overset{f^{\star}\phi  u}{\longrightarrow }f^{\star}s f u\overset{f^{\star}s \alpha }{\longrightarrow }f^{\star}s s\overset{f^{\star}\mu }{\longrightarrow}f^{\star}s .
\end{equation}

\end{theorem}

\begin{proof} 
Without yet knowing that $\pi^{-1}$ as given by (\ref{pi-inv}) is inverse to $\pi$, we calculate $\pi \pi^{-1}$: 
\begin{eqnarray*}
&& \mu u \cdot t \upsilon \cdot tu \varepsilon_f s \cdot \beta f^{\star} s \cdot f^{\star} \mu \cdot f^{\star} s \alpha \cdot f^{\star} \phi u \cdot \eta_f tu \\
& = & \mu u \cdot t \upsilon \cdot tu \mu \cdot tus \alpha \cdot tu \phi u \cdot tu \varepsilon_f ftu \cdot tuf \eta_f tu \cdot \beta tu \\
& = & \mu u \cdot t \upsilon \cdot tu \mu \cdot tus \alpha \cdot tu \phi u \cdot \beta tu \\
& = & \mu u \cdot \mu tu \cdot tt \upsilon \cdot t \upsilon s \cdot tus \alpha \cdot tu \phi u \cdot \beta tu \\
& = & \mu u \cdot t \upsilon \cdot tu \alpha \cdot \mu ufu \cdot t \upsilon fu \cdot tu \phi u \cdot \beta tu \\
& = & \mu u \cdot t \upsilon \cdot tu \alpha \cdot \mu ufu \cdot t \beta u \\
& = & \mu u \cdot t \upsilon \cdot \mu us  \cdot ttu \alpha \cdot t \beta u \\
& = & \mu u \cdot t \mu u \cdot tt \upsilon  \cdot ttu \alpha \cdot t \beta u \\
& = & \mu u \cdot t \eta u \\
& = & 1_{tu}  .
\end{eqnarray*}
The first, fourth, sixth and seventh equalities above follow purely from properties of composition in $\bk$. The second equality uses the triangular equation appropriate to the unit and counit for $f$ and its right adjoint. The third equality uses the opmorphism property of $(u,\upsilon)$ and associativity of $\mu$. The fifth equality uses (\ref{beta}). The eighth equality uses (\ref{triangle2}).   

Now we calculate $\pi^{-1} \pi $: 
\begin{eqnarray*}
&& f^{\star} \mu \cdot f^{\star} s \alpha \cdot f^{\star} \phi u \cdot \eta_f tu \cdot \mu u \cdot t \upsilon \cdot tu \varepsilon_f s \cdot \beta f^{\star} s \\
& = & f^{\star} \mu \cdot f^{\star} s \mu \cdot f^{\star} ss \alpha \cdot f^{\star} s \phi u \cdot f^{\star} \phi tu \cdot \eta_f ttu \cdot t \upsilon \cdot tu \varepsilon_f s  \cdot \beta f^{\star} s \\
& = & f^{\star} \mu \cdot f^{\star} \mu s \cdot f^{\star} s \alpha s \cdot f^{\star} \phi us \cdot \eta_f tus \cdot tu \eta_f s \beta f^{\star} s \\
& = & f^{\star} \mu \cdot f^{\star} s \varepsilon_f s \cdot f^{\star} \eta f f^{\star} s \cdot \eta_f f^{\star}s \\
& = & 1_{f^{\star} s}  ,
\end{eqnarray*}
using the associativity and unit conditions for the monads, the opmorphism property of $(f,\phi)$, equation (\ref{alpha}), and equation (\ref{triangle2}).   \qed

\end{proof}

As expected by general principles of doctrinal adjunction [\ref{Kelly1974}], a monad opmorphism $(f,\phi) : (X,t) \longrightarrow (A,s)$ for which $f$ is a map in $\bk$ gives rise to a monad morphism $(f^{\star},\hat{\phi}) : (A,s) \longrightarrow (X,t)$ where $\hat{\phi} : tf^{\star} \longrightarrow f^{\star}s$ is the mate of $\phi$ under the adjunction $f\dashv f^{\star }$ in the sense of [\ref{KelSt1974}].

\section{Cores between monads}\label{cbm} 

\begin{definition}
An {\it adjunction core} $(u,g,\pi)$ between monads $(A,s)$ and $(X,t)$ in a bicategory $\bk$
consists of the following data in $\bk$:
\begin{enumerate}
\item morphisms $u:A \longrightarrow X$ and $g:A \longrightarrow X$; \ \ 
\item an invertible 2-cell $\pi : gs \longrightarrow tu$. 
\end{enumerate}
\end{definition}

Given such a core, we make the following definitions:

\vspace*{5mm}

(a)  $\bar{\beta} : g \longrightarrow tu$ is the composite $$g \overset{g \eta} \longrightarrow gs \overset{\pi}\longrightarrow tu ;$$

(b)  $\bar{\alpha} : u \longrightarrow gs$ is the composite $$u \overset{\eta u} \longrightarrow tu \overset{\pi^{-1}}\longrightarrow gs ;$$

(c)  $\upsilon : us \longrightarrow tu$ is the composite $$us \overset{\bar{\alpha} s} \longrightarrow gss \overset{g \mu} \longrightarrow gs \overset{\pi} \longrightarrow tu ;$$

(d)   $\psi : tg \longrightarrow gs$ is the composite $$tg \overset{t \bar{\beta}} \longrightarrow ttu \overset{\mu u} \longrightarrow tu \overset{\pi^{-1}} \longrightarrow gs .$$

\begin{proposition}\label{AdjimpliesCore}
An adjunction core between monads $(A,s)$ and $(X,t)$ is obtained from the data of Theorem \ref{ThmAdjKL} by putting $g=f^{\star}$. Moreover, the $\bar{\beta}$ of (a) and the $\bar{\alpha}$ of (b) are the mates of the unit $\beta$ and counit $\alpha$, respectively, the composite in (c) recovers $\upsilon$, and the $\psi$ of (d) is the mate $\hat{\phi}$ of $\phi$.  
\end{proposition}

\begin{proof} 
That we have an adjunction core follows from the invertibility of $\pi$ according to Theorem \ref{ThmAdjKL}.
Next we look at the composite in (a):
\begin{eqnarray*}
&& \pi \cdot f^{\star} \eta \\
& = & \mu u \cdot t \upsilon \cdot tu \varepsilon_f s \cdot \beta f^{\star} s \cdot f^{\star} \eta \\
& = & \mu u \cdot t \upsilon \cdot tu \eta \cdot tu \varepsilon_f \cdot \beta f^{\star} \\
& = & \mu u \cdot t \eta u \cdot tu \varepsilon_f \cdot \beta f^{\star} \\
& = & tu \varepsilon_f \cdot \beta f^{\star} 
\end{eqnarray*}
which is the mate $\bar{\beta}$ of $\beta$.
That the composite in (b) gives $\bar{\alpha}$ is a similar calculation.     

Next we calculate: 
\begin{eqnarray*}
&& \pi^{-1} s \cdot \eta us \\
& = & f^{\star} \mu s \cdot f^{\star} s \alpha s \cdot f^{\star} \phi us \cdot \eta_f tus \cdot \eta us\\
& = & f^{\star} \mu s \cdot f^{\star} s \alpha s \cdot f^{\star} \phi us \cdot f^{\star} f \eta_f us \cdot \eta_f us \\
& = & f^{\star} \mu s \cdot f^{\star} s \alpha s \cdot f^{\star} \eta fus \cdot \eta_f us \\
& = & f^{\star} \mu s \cdot f^{\star} \eta ss \cdot f^{\star} \alpha s \cdot \eta_f us \\
& = & f^{\star} \alpha s \cdot \eta_f us ,
\end{eqnarray*}
so the composite in (c) is:   
\begin{eqnarray*}
&& \mu u \cdot t \upsilon \cdot tu \varepsilon_f s \cdot \beta f^{\star} sa \cdot f^{\star} \mu \cdot \pi^{-1} s \cdot \eta us \\
& = & \mu u \cdot t \upsilon \cdot tu \varepsilon_f s \cdot \beta f^{\star} s \cdot f^{\star} \mu \cdot f^{\star} \alpha s \cdot \eta_f us \\
& = & \mu u \cdot t \upsilon \cdot tu \varepsilon_f s \cdot tuf f^{\star} \mu \cdot \beta f^{\star} ss \cdot f^{\star} \alpha s \cdot \eta_f us \\
& = & \mu u \cdot t \upsilon \cdot tu \varepsilon_f s \cdot tuf f^{\star} \mu \cdot tuf f^{\star} \alpha s \cdot \beta f^{\star} fus \cdot \eta_f us \\
& = & \mu u \cdot t \upsilon \cdot tu \mu \cdot tu \varepsilon_f ss \cdot tuf f^{\star} \alpha s \cdot \beta f^{\star} fus \cdot \eta_f us \\
& = & \mu u \cdot t \mu u \cdot tt \upsilon \cdot t \upsilon s \cdot tu \varepsilon_f ss \cdot tuf f^{\star} \alpha s \cdot \beta f^{\star} fus \cdot \eta_f us \\
& = & \mu u \cdot t \mu u \cdot tt \upsilon \cdot t \upsilon s \cdot tu \alpha s \cdot tu \varepsilon_f fus \cdot \beta f^{\star} fus \cdot \eta_f us \\
& = & \mu u \cdot t \mu u \cdot tt \upsilon \cdot t \upsilon s \cdot tu \alpha s \cdot tu \varepsilon_f fus \cdot tuf \eta_f s \cdot \beta us \\
& = & \mu u \cdot t \mu u \cdot tt \upsilon \cdot t \upsilon s \cdot tu \alpha s \cdot \beta us \\
& = & \mu u \cdot t \upsilon \cdot \mu us \cdot t \upsilon s \cdot tu \alpha s \cdot \beta us \\
& = & \mu u \cdot t \upsilon \cdot \eta us \\
& = & \mu u \cdot \eta tu \cdot \upsilon \\
& = & \upsilon ,
\end{eqnarray*}   
as required. The calculation for the composite in (d) is similar.   \qed
\end{proof}

The Corollary of the following result should be compared with Theorem \ref{enrichedCoreThm}.

\begin{theorem} \label{Core-implies-lots}
For an adjunction core $(u,g,\pi)$ between monads $(A,s)$ and $(X,t)$ in a bicategory $\bk$, the following two commutativity conditions (\ref{pi-cond1}) and (\ref{pi-cond2}) are equivalent.

\begin{equation}\label{pi-cond1}
\xymatrix{
& tus \ar[rd]^-{ t \upsilon}  & \\
gss \ar[ru]^-{\pi s} \ar[d]_-{g \mu} & & ttu \ar[d]^-{\mu u} \\
g s \ar[rr]_-{\pi} & & tu }
\end{equation}

\begin{equation}\label{pi-cond2}
\xymatrix{
& tgs \ar[rd]^-{\psi s}  & \\
ttu \ar[ru]^-{t \pi^{-1}} \ar[d]_-{\mu u} & & gss \ar[d]^-{g \mu} \\
tu \ar[rr]_-{\pi^{-1}} & & gs } 
\end{equation}

Moreover, under these conditions, using definitions (a), (b), (c) and (d),

(i) $(u, \upsilon) : (A,s) \longrightarrow (X,t)$ is a monad opmorphism;

(ii) $(g, \psi) : (A,s) \longrightarrow (X,t)$ is a monad morphism;

(iii) $\pi$ is equal to the composite $$gs \overset{\bar{\beta} s} \longrightarrow tus \overset{t \upsilon} \longrightarrow ttu \overset{\mu u} \longrightarrow tu ;$$ 

(iv) $\pi^{-1}$ is equal to the composite $$tu \overset{t \bar{\alpha}} \longrightarrow tgs \overset{\psi s} \longrightarrow gss \overset{g \mu} \longrightarrow gs ;$$ 

(v) the following identity holds $$(us \overset{\bar{\alpha} s} \longrightarrow gss \overset{g \mu} \longrightarrow gs) = (us \overset{\upsilon} \longrightarrow tu \overset{t \bar{\alpha}} \longrightarrow tgs \overset{\psi s} \longrightarrow gss \overset{g \mu} \longrightarrow gs) ;$$

(vi) the following identity holds $$(tg \overset{t \bar{\beta}} \longrightarrow ttu \overset{\mu u} \longrightarrow tu) = (tg \overset{\psi} \longrightarrow gs \overset{\bar{\beta} s} \longrightarrow tus \overset{t \upsilon} \longrightarrow ttu \overset{\mu u} \longrightarrow tu) ;$$

(vii) the following identity holds $$(g \overset{g \eta} \longrightarrow gs) = (g \overset{\bar{\beta}} \longrightarrow tu \overset{t \bar{\alpha}} \longrightarrow tgs \overset{\psi s} \longrightarrow gss \overset{g \mu} \longrightarrow gs); $$

(viii) the following identity holds $$(u \overset{\eta u} \longrightarrow tu) = (u \overset{\bar{\alpha}} \longrightarrow gs \overset{\bar{\beta} s} \longrightarrow tus \overset{t \upsilon} \longrightarrow ttu \overset{\mu u} \longrightarrow tu). $$

\end{theorem}

\begin{proof} 
Assuming (\ref{pi-cond1}) at the first step, we have the calculation: 
\begin{eqnarray*}
&& \pi \cdot g \mu \cdot \psi s \\
& = & \mu u \cdot t \upsilon \cdot \pi s \cdot \psi s \\
& = & \mu u \cdot t \upsilon \cdot \pi s \cdot \pi^{-1} s \cdot \mu us \cdot t \bar{\beta} s \\
& = & \mu u \cdot t \upsilon \cdot \mu us \cdot t \bar{\beta} s \\
& = & \mu u \cdot \mu tu \cdot tt \upsilon \cdot t \bar{\beta} s \\
& = & \mu u \cdot t \mu u \cdot tt \upsilon \cdot t \bar{\beta} s \\
& = & \mu u \cdot t \pi  ,
\end{eqnarray*}
proving (\ref{pi-cond2}). The converse is dual.   

(i) Using (\ref{pi-cond1}) at the second step, we have the calculation:
\begin{eqnarray*}
&& \mu u \cdot t \upsilon \cdot \upsilon s  \\
& = &  \mu u \cdot t \upsilon \cdot \pi s \cdot g \mu s \cdot \bar{\alpha} ss  \\
& = &  \pi \cdot g \mu \cdot g \mu s \cdot \bar{\alpha} ss  \\
& = &  \pi \cdot g \mu \cdot gs \mu \cdot \bar{\alpha} ss  \\
& = &  \pi \cdot g \mu \cdot \bar{\alpha} s \cdot  u \mu  \\
& = & \upsilon \cdot u \mu .
\end{eqnarray*}  
We also have:
\begin{eqnarray*}
&& \upsilon \cdot u \eta  \\
& = &  \pi \cdot g \mu \cdot \bar{\alpha} s \cdot u \eta   \\
& = &  \pi \cdot g \mu \cdot \pi^{-1} s \cdot \eta us \cdot u \eta  \\
& = &  \pi \cdot g \mu \cdot \pi^{-1} s \cdot tu \eta \cdot \eta u   \\
& = &  \pi \cdot g \mu \cdot gs \eta \cdot \pi^{-1} \cdot \eta u  \\
& = &  \pi \cdot \pi^{-1} \cdot \eta u  \\
& = &   \eta u .
\end{eqnarray*}  
Hence $(u,\upsilon)$ is a monad opmorphism.

(ii) This is dual to (i) using (\ref{pi-cond2}) instead of (\ref{pi-cond1}). 

(iii) Using (\ref{pi-cond1}), we have:
\begin{eqnarray*}
&& \pi \\
& = & \pi \cdot g \mu \cdot g \eta s \\
& = &  \mu u \cdot t \upsilon \cdot \pi s \cdot g \eta s \\
& = &  \mu u \cdot t \upsilon \cdot \bar{\beta} s .
\end{eqnarray*}

(iv) This is dual to (iii) using (\ref{pi-cond2}) instead of (\ref{pi-cond1}). 

(v) Using (iv), we immediately have:
\begin{eqnarray*}
&& g \mu \cdot \psi s \cdot t \bar{\alpha} \cdot \upsilon \\
& = & \pi^{-1} \cdot \upsilon \\
& = & \pi^{-1} \cdot \pi \cdot g \mu \cdot \bar{\alpha} s \\
& = & g \mu \cdot \bar{\alpha} s .
\end{eqnarray*}

(vi) This is dual to (v) using (iii) instead of (iv).

(vii) Using (iv), we immediately have:
\begin{eqnarray*}
&& g \mu \cdot \psi s \cdot t \bar{\alpha} \cdot \bar{\beta} \\
& = & \pi^{-1} \cdot \bar{\beta} \\
& = & \pi^{-1} \cdot \pi \cdot g \eta \\
& = & g \eta .
\end{eqnarray*}

(viii) This is dual to (vii) using (iii) instead of (iv).   \qed

\end{proof}

\begin{corollary} \label{KLCoreThm}
An adjunction core of the form $(u,f^{\star},\pi)$ between monads $(A,s)$ and $(X,t)$ in a bicategory $\bk$  extends to an adjunction (\ref{AdjKL}) in $\mathrm{KL}(\bk)$ if and only if one of the diagrams (\ref{pi-cond1}) or (\ref{pi-cond2}) commutes. The adjunction is unique when it exists.
\end{corollary}

\begin{proof} 
Properties (i)--(viii) of Theorem \ref{Core-implies-lots}, when re-expressed with the mates $\alpha$, $\beta$ and $\phi$ replacing  $\bar{\alpha}$, $\bar{\beta}$ and $\psi$, give precisely what is required for an adjunction (\ref{AdjKL}). The converse is Proposition \ref{AdjimpliesCore}.  \qed
\end{proof}

\section{Cores between internal categories}\label{cbic}

This section will apply our results to categories internal to a category $\bc$ which admits pullbacks. For this example, we take the bicategory $\bk$ of the previous sections to be bicategory $\mathrm{Span}(\bc)$ of spans in $\bc$ as constructed by B\'enabou in [\ref{Ben1967}].   

The objects of the bicategory $\mathrm{Span}(\bc)$ are those of $\bc$. A morphism $S = (s_0, S, s_1) : U \longrightarrow V$ is a so-called span $$U \overset{s_0} \longleftarrow S \overset{s_1} \longrightarrow V$$ from $U$ to $V$ in $\bc$. A 2-cell $r : (s_0, S, s_1) \longrightarrow (t_0, T, t_1) : U \longrightarrow V$ is a morphism $r:S \longrightarrow T$ in $\bc$ such that $t_0 r = s_0$ and $t_1 r = s_1$. Vertical composition of 2-cells is simply that of $\bc$. Horizontal composition uses pullback in $\bc$; more precisely, $$(U \overset{(s_0,S,s_1)} \longrightarrow V \overset{(t_0,T,t_1)} \longrightarrow W) = (U \overset{(s_0 p,P,t_1 q)} \longrightarrow W)$$ where
\begin{equation}
\xymatrix{
P \ar[rr]^-{q} \ar[d]_-{p} && T \ar[d]^-{t_0} \\
S \ar[rr]_-{s_1} && V}
\end{equation}  
is a pullback square.   

Each morphism $f:U \longrightarrow V$ in $\bc$ determines a span $f_{\star} = (1_U, U, f) : U \longrightarrow V$. We write $f^{\star} : V \longrightarrow U$ for the span $(f, V, 1_V) : V \longrightarrow U$. It is well known that we have an adjunction $f_{\star} \dashv f^{\star}$ in $\mathrm{Span}(\bc)$; in fact, it is shown in [\ref{CKS}] that the maps in $\mathrm{Span}(\bc)$ are all isomorphic to spans of the form $f_{\star} : U \longrightarrow V$ for some $f:U \longrightarrow V$ in $\bc$. 

One of the reasons for interest in the free Kleisli object cocompletion $\mathrm{KL}(\bk)$ in the paper [\ref{FTMII}] is that the 2-category $\mathrm{Cat}(\bc)$ of categories in $\bc$ is equivalent to the sub-2-category of $\mathrm{KL}(\mathrm{Span}(\bc))$ obtained by restricting to the morphisms whose underlying morphisms in $\mathrm{Span}(\bc)$ are maps. We shall explain this in a bit more detail.   

A \textit{category} in $\bc$ is a monad $(A,S)$ in $\mathrm{Span}(\bc))$. The object $A$ of $\bc$ is called the \textit{object of objects}. The span $S = (s_0, S, s_1) : A \longrightarrow A$ provides the \textit{object of morphisms} $S$ and the \textit{source and target operations} $s_0$ and $s_1$. The multiplication for the monad provides the \textit{composition operation} and the unit for the monad provides the \textit{identities operation}.      

A \textit{functor} between categories in $\bc$ is a monad opmorphism of the form $(f_{\star},\phi) : (X,T) \longrightarrow (A,S)$ in $\mathrm{Span}(\bc))$. The morphism $f:X \longrightarrow A$ in $\bc$ is called the \textit{effect on objects} of the functor and the morphism $\phi : T \longrightarrow S$ in $\bc$ is called the \textit{effect on morphisms} of the functor.    

A \textit{natural transformation} between functors in $\bc$ is precisely a 2-cell between them in $\mathrm{KL}(\mathrm{Span}(\bc))$.   

As an immediate consequence of Corollary \ref{KLCoreThm} we have:

\begin{corollary} \label{IntCatCoreThm}
An adjunction core of the form $(u_{\star},f^{\star},\pi)$ between categories $(A,S)$ and $(X,T)$ in a category $\bc$  extends to an adjunction $(f_{\star}, \phi)  \dashv (u_{\star}, \upsilon)$ in $\mathrm{Cat}(\bc)$ if and only if one of the diagrams (\ref{pi-cond1}) or (\ref{pi-cond2}) \emph{mutatis mutandis} commutes. The adjunction is unique when it exists. \qed
\end{corollary}

\section{A doctrinal setting}\label{doctrinal}   

The idea of adapting the 2-cells of $\mathrm{KL}(\bk)$ or $\mathrm{EM}(\bk)$ to the doctrinal setting was recently exposed by Mark Weber [\ref{WeberACS}].

Let $D$ be a pseudomonad (also called a doctrine in [\ref{Law1969}], [\ref{Kelly1974}], [\ref{Zoeb}] and [\ref{FiB}]) on a bicategory $\bk$. It means that we have a pseudofunctor $D : \bk \longrightarrow \bk$, a unit pseudonatural transformation denoted by $n : 1_{\bk} \longrightarrow D$, and a multiplication pseudonatural transformation denoted by  $m : DD \longrightarrow D$. For example, see [\ref{Marm}] [\ref{Lack2000}] for the axioms. 

A \textit{lax $D$-algebra} $(A,s)$ consists of an object $A$, a morphism $s:DA \longrightarrow A$, and 2-cells $\mu : s \cdot Ds \Longrightarrow s \cdot m_A$ and  $\eta : 1_A \Longrightarrow s \cdot n_A$, satisfying coherence conditions.  

For lax $D$-algebras $(A,s)$ and $(A^{\prime },s^{\prime })$ in $\bk$, a {\it lax opmorphism} $(f,\phi):(A,s) \longrightarrow (A^{\prime },s^{\prime })$ consists of a morphism $f:A \longrightarrow A^{\prime }$ and a 2-cell $\phi: fs \Longrightarrow s^{\prime } Df$ in $\bk$ such that

\begin{equation} \label{laxopmor1}
\xymatrix{
& s^{\prime}Df \cdot Ds \ar[r]^-{s^{\prime} D\phi} & s^{\prime} Ds^{\prime} \cdot D^{2} f \ar[rd]^-{\mu D^{2} f} & \\
fsDs \ar[ru]^-{\phi Ds} \ar[rd]_-{f \mu} & & & s^{\prime} m_{A^{\prime }} D^{2} f \\
& fsm_A \ar[r]_-{\phi m_A} & s^{\prime} Df \cdot m_A \ar[ru]_-{s^{\prime} m_f} &}
\end{equation}

\begin{equation} \label{laxopmor2}
\xymatrix{
fsn_A \ar[rr]^-{\phi n_A}  && s^{\prime} Df \cdot n_A \ar[d]^-{s^{\prime} n_f} \\
f \ar[rr]_-{\eta f} \ar[u]_-{f \eta} && s^{\prime} n_{A^{\prime}}}
\end{equation}

\noindent The composite of lax opmorphisms 
$(f,\phi):(A,s) \longrightarrow (A^{\prime },s^{\prime })$ 
and $(f^{\prime },\phi ^{\prime }):(A^{\prime },s^{\prime }) \longrightarrow (A^{\prime\prime },s^{\prime\prime })$ 
is defined to be $(f^{\prime }f,\phi ^{\prime } \star \phi):(A,s) \longrightarrow (A^{\prime\prime },s^{\prime\prime })$ 
where $\phi ^{\prime } \star \phi$ is the composite 
\begin{equation}
f^{\prime} f s \overset{f^{\prime} \phi} \longrightarrow f^{\prime}
s^{\prime} Df \overset{\phi^{\prime} Df} \longrightarrow
s^{{\prime\prime}} Df^{\prime} \cdot Df \ \ .
\end{equation}
 
There is a bicategory $\mathrm{KL}(\bk,D)$ whose objects are lax $D$-algebras and whose morphisms are lax opmorphisms. A 2-cell $\rho : (f, \phi)\Longrightarrow (g, \psi) : (A,s) \longrightarrow (A^{\prime},s^{\prime})$ in $\mathrm{KL}(\bk,D)$ (in non-reduced form, in the terminology of \ref{FTMII}) is a 2-cell $\hat{\rho} : f s \Longrightarrow s^{\prime} Dg$ in $\bk$ such that the diagrams (\ref{KLD2cell1}) and (\ref{KLD2cell2}) commute.  

\begin{equation}\label{KLD2cell1}
\xymatrix{
& s^{\prime} Dg \cdot Ds \ar[r]^-{s^{\prime} D\psi} & s^{\prime} Ds^{\prime} D^{2} g \ar[rd]^-{\mu D^{2} g} & \\
 fsDs \ar[ru]^-{\hat{\rho} Ds} \ar[rd]_-{f \mu} & & & s^{\prime} m_{A^{\prime}} D^{2} g\\
& fsm_A \ar[r]_-{\hat{\rho} m_A} & s^{\prime} Dg \cdot m_A \ar[ru]_-{s^{\prime} m_g}  &}
\end{equation}

\begin{equation}\label{KLD2cell2}
\xymatrix{
& s^{\prime} Df \cdot Ds \ar[r]^-{s^{\prime} D \hat{\rho}} & s^{\prime} Ds^{\prime} D^{2} g \ar[rd]^-{\mu D^{2} g} & \\
 fsDs \ar[ru]^-{\phi Ds} \ar[rd]_-{f \mu} & & & s^{\prime} m_{A^{\prime}} D^{2} g\\
& fsm_A \ar[r]_-{\hat{\rho} m_A} & s^{\prime} Dg \cdot m_A \ar[ru]_-{s^{\prime} m_g}  &}
\end{equation}
The reduced form of the 2-cell in $\mathrm{KL}(\bk,D)$ is a 2-cell $\rho : f  \Longrightarrow s^{\prime} Dg \cdot n_A$ in $\bk$ such that equality (\ref{KLD2cellred}) holds. 

\begin{multline}\label{KLD2cellred}
f s\overset{\phi } {\longrightarrow} s^{\prime }Df \overset{s^{\prime
}D\rho } {\longrightarrow} s^{\prime} Ds^{\prime } D^{2} g Dn_A \overset{\mu
D^{2} g Dn_A}{\longrightarrow} s^{\prime} m_{A^{\prime}} D^{2} g Dn_A \\ 
\overset{s^{\prime} m_{g}^{-1} Dn_A}{\longrightarrow}s^{\prime} Dg  \cdot m_A Dn_A \overset{\cong}{\longrightarrow} s^{\prime} Dg  \\
= \\
f s\overset{\rho s}{\longrightarrow } s^{\prime }
Dg \cdot n_A s \overset{s^{\prime} n_g s}{\longrightarrow}s^{\prime} n_{A^{\prime}} g s 
\overset{s^{\prime} n_{A^{\prime}} \psi }{\longrightarrow} s^{\prime} n_{A^{\prime}} s^{\prime} Dg \\
\overset{s^{\prime} n_{s^{\prime}}^{-1} Dg}{\longrightarrow} 
s^{\prime} Ds^{\prime} n_{DA^{\prime}} Dg \overset{\mu n_{DA^{\prime}} Dg}{\longrightarrow
}s^{\prime} m_{A^{\prime}} n_{DA^{\prime}} Dg \overset{\cong}{\longrightarrow
}s^{\prime} Dg  .
\end{multline}
The bijection between reduced and non-reduced forms takes $\rho$ to the $\hat{\rho}$ defined by either side of (\ref{KLD2cellred}).  

The vertical composite of the 2-cells $\rho:(f,\phi) \longrightarrow (g,\psi)$ and $\tau:(g,\psi)\longrightarrow (h,\theta)$ is the 2-cell
\begin{equation}
f^{\prime} fs\overset{f^{\prime} \hat{\rho}}{\longrightarrow }f^{\prime} s^{\prime } Dg \overset{s^{\prime
}\tau }{\longrightarrow }s^{\prime }( s^{\prime }h) \overset{\cong
}{\longrightarrow }\left( s^{\prime }s^{\prime }\right) h\overset{\mu
h}{\longrightarrow }s^{\prime }h\ \ .
\end{equation}

The horizontal composite of 2-cells $\rho : (f, \phi)\longrightarrow (g, \psi) : (A,s) \longrightarrow (A^{\prime},s^{\prime})$ and $\rho^{\prime} : (f^{\prime}, \phi^{\prime})\longrightarrow (g^{\prime}, \psi^{\prime}) : (A^{\prime},s^{\prime}) \longrightarrow (A^{\prime \prime},s^{\prime \prime})$ has non-reduced form the 2-cell 
\begin{multline}
f \overset{\rho}{\longrightarrow} s^{\prime} Dg \cdot n_A \overset{s^{\prime} D\tau \cdot n_A}{\longrightarrow } s^{\prime} Ds^{\prime} D^{2}hDn_A \cdot n_A \overset{\mu D^{2}hDn_A \cdot n_A}{\longrightarrow } s^{\prime} m_{A^{\prime}} D^{2}hDn_A \cdot n_A \\
\overset{s^{\prime} m_{h}^{-1}Dn_A \cdot n_A}{\longrightarrow } s^{\prime} Dh \cdot m_{A} Dn_A \cdot n_A\overset{\cong}{\longrightarrow} s^{\prime} Dh \cdot n_A\ \ .
\end{multline}

This completes the definition of the bicategory $\mathrm{KL}(\bk,D)$ with the exception of giving the coherent associativity and unit isomorphisms. As forewarned, we have been writing as if $\bk$ were a 2-category in which case $\mathrm{KL}(\bk,D)$ would also be a 2-category. Putting in all the coherent isomorphisms as we did in the definition of $\mathrm{KL}(\bk)$, we can readily give them for $\mathrm{KL}(\bk, D)$ as we did for the special case where $D$ was the identity pseudomonad. 

In a future paper we shall explain a universal property of the construction taking $(\bk,D)$ to $\mathrm{KL}(\bk , D)$.

\begin{definition}
An {\it adjunction core} $(u,g,\pi)$ between lax $D$-algebras $(A,s)$ and $(X,t)$ in a bicategory $\bk$
consists of the following data in $\bk$:
\begin{enumerate}
\item morphisms $u:A \longrightarrow X$ and $g:A \longrightarrow X$; \ \ 
\item an invertible 2-cell $\pi : gs \longrightarrow t Du$. 
\end{enumerate}
\end{definition} 
Given such a core, we make the following definitions:

\vspace*{5mm}

(a)  $\bar{\beta} : g \Longrightarrow  tDu \cdot n_A$ is the composite $$g \overset{g \eta} \longrightarrow gsn_A \overset{\pi n_A}\longrightarrow tDu \cdot n_A  \ \ ;$$

(b)  $\bar{\alpha} : u \Longrightarrow gsn_A$ is the composite $$u \overset{\eta u} \longrightarrow t n_X u \overset{\cong} \longrightarrow t Du \cdot n_A \overset{\pi^{-1} n_A} {\longrightarrow} gs n_A \ \ ; $$

(c)  $\upsilon : us \Longrightarrow tDu$ is the composite $$us \overset{\bar{\alpha} s} \longrightarrow gsn_A s \overset{gsn_s}{\longrightarrow} gsDs \cdot n_{DA} \overset{g \mu n_{DA}} \longrightarrow g s m_A n_{DA}\overset{\cong} \longrightarrow gs \overset{\pi} \longrightarrow tDu \ \ ;$$

(d)   $\psi : tDg \Longrightarrow gs$ is the composite $$tDg \overset{t D\bar{\beta}} \longrightarrow tDtD^{2} u Dn_A \overset{\mu n_{u}^{-1}} \longrightarrow tm_X Dn_XDu \overset{\cong} \longrightarrow tDu \overset{\pi^{-1}} \longrightarrow gs \ \ .$$

If $g=f^{\star}$ for some map $f : X \longrightarrow A$ in $\bk$ then $\bar{\alpha}$ and $\bar{\beta}$ have mates $\alpha : fu \Longrightarrow sn_A$ and $\beta^{\prime} : 1_X \Longrightarrow  tDu \cdot n_A f$. We take $\beta : 1_X \Longrightarrow  tDuDf \cdot n_X$ to be the composite of $\beta^{\prime}$ and $t n_f : tDu \cdot n_A f \cong tDuDf \cdot n_X$.   

\begin{theorem} \label{KLDCoreThm}
An adjunction core of the form $(u,f^{\star},\pi)$ between lax $D$-algebras $(A,s)$ and $(X,t)$ in a bicategory $\bk$  extends to an adjunction 
\begin{equation}
(f,\phi) \dashv (u,\upsilon) : (A,s) \longrightarrow (X,t)
\end{equation}
with counit $\alpha$ and unit $\beta$ in $\mathrm{KL}(\bk,D)$ if and only if one of the diagrams (\ref{pi-Dcond1}) or (\ref{pi-Dcond2}) commutes. The adjunction is unique when it exists.
\end{theorem}  

\begin{equation}\label{pi-Dcond1}
\xymatrix{
& tDuDs \ar[r]^-{tD\upsilon} & tDtD^2 u \ar[rd]^-{\mu D^2 u} & \\
gsDs \ar[ru]^-{\pi Ds} \ar[rd]_-{g\mu} & & & tm_XD^2 u \\
& gsm_A \ar[r]_-{\pi m_A} & tDu \cdot m_A \ar[ru]_-{t m_u} &}
\end{equation}

\begin{equation}\label{pi-Dcond2}
\xymatrix{
& tDgDs \ar[r]^-{\psi Ds} & gsDs \ar[rd]^-{g \mu} & \\
tDtD^2u \ar[ru]^-{tD\pi^{-1}} \ar[rd]_-{\mu D^2u} & & & gsm_A \\
& tm_XD^2u \ar[r]_-{tm_{u}^{-1}} & tDu \cdot m_A \ar[ru]_-{\pi^{-1}m_A} &}
\end{equation}

\begin{center}
--------------------------------------------------------
\end{center}

\appendix

\end{document}